\documentclass[12pt]{amsart}

\usepackage{hyperref}
\usepackage{amssymb, latexsym, amsthm, enumerate, mathptmx}

\usepackage{tikz}
\usetikzlibrary{arrows}
\usepackage[all]{xy}
\usepackage{color}
\usepackage{cleveref}
\usepackage{mathtools}
\usepackage{paralist}

\newtheorem{theorem}{Theorem}[section]
\newtheorem{lemma}[theorem]{Lemma}
\newtheorem{corollary}[theorem]{Corollary}

\newtheorem{question}[theorem]{Question}
\newtheorem{conjecture}[theorem]{Conjecture}

\newtheoremstyle{definition}% name
  {6pt}%      Space above
  {6pt}%      Space below
  {}%         Body font
  {}%         Indent amount (empty = no indent, \parindent = para indent)
  {\bfseries}% Thm head font
  {.}%        Punctuation after thm head
  {.5em}%     Space after thm head: " " = normal interword space;
  {}%

\theoremstyle{definition}
\newtheorem{definition}[theorem]{Definition}

\newtheoremstyle{remark}% name
  {6pt}%      Space above
  {6pt}%      Space below
  {}%         Body font
  {}%         Indent amount (empty = no indent, \parindent = para indent)
  {\bfseries}% Thm head font
  {.}%        Punctuation after thm head
  {.5em}%     Space after thm head: " " = normal interword space;
  {}%

\theoremstyle{remark}
\newtheorem{remark}[theorem]{Remark}

\newtheoremstyle{example}% name
  {6pt}%      Space above
  {6pt}%      Space below
  {}%         Body font
  {}%         Indent amount (empty = no indent, \parindent = para indent)
  {\bfseries}% Thm head font
  {.}%        Punctuation after thm head
  {.5em}%     Space after thm head: " " = normal interword space;
  {}%

\theoremstyle{example}
\newtheorem{example}[theorem]{Example}

\makeatletter
\renewcommand\@makefntext[1]{%
\setlength\parindent{1em}%
\noindent
\makebox[1.8em][r]{}{#1}}
\makeatother

\title{Flawlessness of $h$-vectors of broken circuit complexes}

\author{Martina Juhnke-Kubitzke}
\address{Institut f\"ur Mathematik, Universit\"at Osnabr\"uck, 49069 Osnabr\"uck, Germany}
\email{juhnke-kubitzke@uos.de}

\author{Dinh Van Le}
\address{Institut f\"ur Mathematik, Universit\"at Osnabr\"uck, 49069 Osnabr\"uck, Germany}
\email{dlevan@uos.de}

\begin{document}

\begin{abstract}
 One of the major open questions in matroid theory asks whether the $h$-vector $(h_0,h_1,\ldots,h_s)$ of the broken circuit complex of a matroid $M$ satisfies the following inequalities:
$$
   h_0\leq h_1\leq \cdots\leq h_{\lfloor s/2\rfloor} \quad \text{and}\quad
   h_i\le h_{s-i}\ \text{ for }\ 0\leq i \leq \lfloor s/2\rfloor.  
$$
This paper affirmatively answers the question for matroids that are representable over a field of characteristic zero.
 \end{abstract}
 
\maketitle

\section{Introduction}

The notion of broken circuit complexes goes back to Whitney \cite{Wh32}, who used his broken circuit idea to interpret the coefficients of the chromatic polynomial of a graph. This notion was later extended to matroids by Rota \cite{Ro64} and Brylawski \cite{Bry77}. Given a loopless matroid $M$ on ground set $E$, which is endowed with a linear ordering $<$, a \emph{broken circuit} of $(M,<)$ is a circuit of $M$ with its least element removed. The \emph{broken circuit complex} of $(M,<)$, denoted by $BC_<(M)$ (or briefly $BC(M)$ if no confusion may arise), is defined by
$$BC(M):=\{F\subseteq E~:~ F \ \text{contains no broken circuit} \}.$$

Broken circuit complexes have shown to be important in multiple ways. From the algebraic point of view, they play an interesting role in the study of hyperplane arrangements. In particular, the broken circuit idea was used to construct bases for two fundamental algebraic objects associated with a hyperplane arrangement, namely, the Orlik--Solomon algebra and the Orlik--Terao algebra \cite{Bj92, PS06}. Through these constructions, broken circuit complexes have been an essential tool for studying important algebraic and homological properties of those algebras \cite{DGT14, EPY03, KR09, Le14, LR13}.

\footnotetext{
	\begin{itemize}
		\item[ ]
		{\it MSC (2010)}: 05B35, 13F55.
		\item[ ]
		{\it Keywords}: Broken circuit complex, flawless, $h$-vector, matroid, unimodal.
	\end{itemize}
}

From the combinatorial point of view, $f$-vectors and $h$-vectors of broken circuit complexes encode very useful information about the underlying matroids. Recall that the \emph{characteristic polynomial} of a matroid $M$ is defined as $\chi(M;t):=\sum_{X\subseteq E}(-1)^{|X|}t^{r(M)-r(X)}$, where $r(\cdot)$ denotes the rank function of $M$. This polynomial, which was introduced by Rota \cite{Ro64} as a generalization of the chromatic polynomial of a graph, plays a prominent role in the study of many combinatorial problems; see, e.g., \cite{BO92,Za87}. A fascinating property of $f$-vectors of broken circuit complexes, which primarily makes these complexes important, is the following formula due to Whitney \cite{Wh32} and Rota \cite{Ro64}:
\begin{equation}\label{WR}
 \chi(M;t)=\sum_{i=0}^r(-1)^if_it^{r-i},
\end{equation}
where $f_i$ denotes the number of faces of $BC(M)$ of cardinality $i$. The $h$-vector of $BC(M)$, on the other hand, encodes the shelling polynomial of $BC(M)$ \cite{Bj92}. Furthermore, several properties of $M$ (such as  connectivity \cite{Cr67} or being  a series--parallel network \cite{Bry71}) and of $BC(M)$ (such as Gorensteinness or being a complete intersection \cite{Le14}) are determined by the $h$-vector of $BC(M)$. 
 For these reasons, $f$-vectors and $h$-vectors of broken circuit complexes are among the most interesting numerical invariants in matroid theory.
Recently, great advances have been made in the study of $f$-vectors and $h$-vectors of broken circuit complexes. In particular, the long-standing conjectures of Rota--Heron \cite{Ro71, He72} and Welsh \cite{We76} on the unimodality and log-concavity of the $f$-vector of $BC(M)$ have been resolved by Adiprasito, Huh and Katz \cite{AHK}. Additionally, Huh \cite{Hu15} proved that the $h$-vector of $BC(M)$ is log-concave if $M$ is representable over a field of characteristic zero. Recall that a sequence $(a_0,a_1,\ldots,a_n)$ of real numbers is said to be \emph{log-concave} if $a_j^2\ge a_{j-1}a_{j+1}$ for all $1\le j\le n-1$. Also, this sequence is called \emph{unimodal} if there exists $0\leq p\leq n$ such that $a_0\leq a_1\leq \cdots \leq a_p\geq a_{p+1}\geq \cdots \geq a_n$. Observe that if a sequence of positive numbers is log-concave, then it is unimodal.

 Despite the significant advances mentioned above, $f$-vectors and $h$-vectors of broken circuit complexes are still rather mysterious. In fact, the problem of characterizing these vectors is widely regarded as out of reach at the moment. A more realistic problem would be to find as many restrictions on these vectors as possible.

Such restrictions are predicted by the following conjecture, which is in the focus of this paper:

\begin{conjecture}\label{con}
 Let $M$ be a loopless matroid. Let $(h_0,h_1,\ldots,h_s)$ be the $h$-vector of $BC(M)$, where $s$ is the largest index $j$ with $h_j\ne0$. Then the following inequalities hold:
\[ h_0\leq h_1\leq \cdots\leq h_{\lfloor s/2\rfloor} \quad \text{and}\quad h_i\le h_{s-i}\ \text{ for }\ 0\leq i \leq \lfloor s/2\rfloor.  \]
\end{conjecture}

A sequence $(h_0,h_1,\ldots,h_s)$ of real numbers that satisfies the inequalities in the above conjecture is called \emph{strongly flawless}, and it is called \emph{flawless} if $h_i\le h_{s-i}\ \text{ for }\ 0\leq i \leq \lfloor s/2\rfloor$. Clearly, the strongly flawless condition can be rephrased as $h_i\le h_j$ for $0\le i\le j\le s-i$. Moreover, for a unimodal sequence, being flawless is equivalent to being strongly flawless.

\Cref{con} goes back to a still wide open conjecture of Stanley \cite{Sta77}, which anticipates that the $h$-vector of the \emph{independence complex} $IN(M)$ of a matroid $M$ is a pure $O$-sequence. The reader is referred to \cite{BMM+} for the definition of pure $O$-sequences as well as recent developments in the study of these interesting objects. Recall that $IN(M)$ is the collection of all independent sets in $M$, and that it contains $BC(M)$ as a subcomplex. In \cite{Hi89}, Hibi showed that a pure $O$-sequence is strongly flawless. Inspired by this result, he proposed a weaker version of Stanley's conjecture in \cite{Hi92}, predicting that the $h$-vector of $IN(M)$ must be strongly flawless. This conjecture was resolved by Chari \cite{Ch97}, who proved that $IN(M)$ has a \emph{convex ear decomposition}. Subsequently, an algebraic version of Chari's proof, which shows the existence of \emph{$g$-elements} for a general Artinian reduction of the Stanley--Reisner ring of $IN(M)$, was given by Swartz in \cite{Sw03}. Therein, \Cref{con} was also mentioned implicitly. As the set of $h$-vectors of independence complexes is strictly contained in the set of $h$-vectors of broken circuit complexes (see \cite{Bry77}), \Cref{con} is stronger than and, in particular, implies Hibi's conjecture. It is worth emphasizing that the techniques of Chari and Swartz for proving Hibi's conjecture do not work in the case of broken circuit complexes, and thus cannot be used to establish \Cref{con}. Indeed, Swartz \cite{Sw03} provided examples of matroids whose broken circuit complexes do not admit $g$-elements and hence also fail to have a convex ear decomposition.

The main goal of this paper is to verify \Cref{con} for matroids representable over a field of characteristic zero. In fact, we prove a somewhat stronger result. We say that a class of matroids $\mathcal{M}$ has a certain property (such as {unimodal} or strongly flawless) if the $h$-vector of the broken circuit complex of every matroid in $\mathcal{M}$ has that property. The main result of this paper is as follows. 

\begin{theorem}\label{th32}
 Let $\mathcal{M}$ be a minor-closed class of matroids. If $\mathcal{M}$ is unimodal, then it is strongly flawless.
\end{theorem}

Note that this theorem implies \Cref{con} for matroids representable over a field of characteristic zero, by virtue of Huh's log-concavity result \cite{Hu15} (see \Cref{co34}).

Let us briefly outline how the proof of \Cref{th32} proceeds. As mentioned before, a unimodal, flawless sequence is also strongly flawless. So it suffices to show that the $h$-vector of $BC(M)$ is flawless for every matroid $M\in\mathcal{M}$. To this end, we first reduce the proof to the case where $M$ is minimally connected (see \Cref{red}). In this case, $M$ contains a removable series class $S$ (see \Cref{rsc}). We then find two different ways to relate the $h$-vector of $BC(M)$ to the $h$-vector of $BC(M/S)$ (see \Cref{series1,series2}). Combining these comparisons, the flawlessness of the $h$-vector of $BC(M)$ will follow by induction and the unimodality of the $h$-vector of $BC(M/S)$.

This paper is organized as follows. In the next section, we review the basic notions of matroids and broken circuit complexes. \Cref{sec3} contains the proof of \Cref{th32} 
and its immediate application to Orlik--Terao algebras. Finally, some questions related to our work are discussed in \Cref{sec4}.

\section{Preliminaries}

\subsection{Matroids}

The notion of matroids was introduced by Whitney \cite{Wh35} as a common generalization of dependence in linear algebra and graph theory. Since then a rich theory of matroids has been developed which provides a framework for approaching many combinatorial problems. In the following, we collect the needed facts and definitions from matroid theory, referring to the seminal book by Oxley \cite{Ox11} for more details.

\begin{definition}
 A \emph{matroid} $M=(E,\mathcal{I})$ consists of a finite ground set $E$ and a nonempty collection $\mathcal{I}$ of subsets of $E$, called \emph{independent sets}, satisfying the following conditions:
\begin{enumerate}
\item
If $I\in\mathcal{I}$ and $J\subseteq I$, then $J\in\mathcal{I}$.
\item
If $I,I'\in\mathcal{I}$ and $|I|<|I'|$, then there exists $e\in I'-I$ such that $I\cup e\in\mathcal{I}$.
\end{enumerate}
\end{definition}

In a matroid $M=(E,\mathcal{I})$, a \emph{basis} is a maximal independent set. A subset of $E$ is called \emph{dependent} if it is not a member of $\mathcal{I}$. A \emph{circuit} is a minimal dependent set, and an \emph{$m$-circuit} is a circuit of cardinality $m$. For any set $X\subseteq E$, all maximal independent subsets of $X$ have the same size, which is called the \emph{rank} $r(X)$ of $X$. In particular, the rank of $E$, which is the common cardinality of all the bases of $M$, is also called the \emph{rank} of $M$ and denoted by $r(M)$. A matroid can be specified by either its collection of bases, its collection of circuits, or its rank function. In fact, there are equivalent definitions of matroids in terms of bases, circuits, and rank functions.

Two matroids $M=(E,\mathcal{I})$ and $M'=(E',\mathcal{I}')$ are \emph{isomorphic} if there exists a bijection $\varphi:E\to E'$ such that for every subset $X$ of $E$, $X\in\mathcal{I}$ if and only if $\varphi(X)\in\mathcal{I}'$.

The prototypical example of a matroid is the \emph{vector matroid} $M[A]$ of a matrix $A$: the ground set $E$ of $M[A]$ is taken to be the set of columns of $A$, and a subset $I\subseteq E$ is independent if and only if the corresponding columns are linearly independent. A matroid is \emph{representable} over a field $K$ if it is isomorphic to the vector matroid of a matrix over $K$. It should be noted, however, that not every matroid is representable over some field; see \cite[Proposition 6.1.10]{Ox11}.

Let $M$ be a matroid on the ground set $E$. Let $\mathcal{B}$ be the collection of bases of $M$. Then $\mathcal{B}^*=\{E-B~:~ B\in\mathcal{B}\}$ is also the collection of bases of a matroid $M^*$. We call this matroid the \emph{dual} of $M$. For example, $M[A]^*\cong M[A^*]$ for any matrix $A$, where $A^*$ is a matrix whose row space is the orthogonal space of the row space of $A$.

An element $e\in E$ is called a \emph{loop} if $\{e\}$ is a circuit of $M$. We say that $M$ is \emph{loopless} if it has no loops. A loop of $M^*$ is called a \emph{coloop} of $M$. More generally, circuits of $M^*$ are  called \emph{cocircuits} of $M$. A \emph{series class} $S$ of $M$ is a maximal subset of $E$ such that $S$ contains no coloops and if $e,f$ are distinct elements of $S$, then $\{e,f\}$ is a cocircuit of $M$. A series class is \emph{non-trivial} if it contains at least two elements. Notice that if $S$ is a series class and $C$ is a circuit of $M$, then either $C\cap S=\emptyset$ or $S\subseteq C$. This follows from the well-known fact that a circuit and a cocircuit of $M$ cannot have just a single element in common; see \cite[Proposition 2.1.11]{Ox11}.

Let $X$ be a subset of $E$. The \emph{deletion} of $X$ from $M$, denoted $M-X$, is the matroid on ground set $E-X$ whose independent sets are the independent sets of $M$ that are contained in $E-X$. The \emph{contraction} of $X$ from $M$ is defined to be $M/X=(M^*-X)^*$. Note that the operations of deletion and contraction commute, i.e., $(M-X)/Y=M/Y-X$ for disjoint subsets $X$ and $Y$ of $E$. 
A \emph{minor} of $M$ is a matroid which can be obtained from $M$ by a sequence of deletions and contractions. 
A class of matroids $\mathcal{M}$ is said to be \emph{minor-closed} if for every $M\in\mathcal{M}$, all minors of $M$ are also members of $\mathcal{M}$.

Let $M_1$ and $M_2$ be matroids on disjoint ground sets $E_1$ and $E_2$. Their \emph{direct sum} $M_1\oplus M_2$ is the matroid on ground set $E_1\cup E_2$ 
whose independent sets are all possible unions of an independent set of $M_1$ with an independent set of $M_2$. The direct sum of a finite collection of matroids is then defined by iterating the previous construction. 
A matroid is called \emph{connected} if it is not the direct sum of two smaller matroids. Otherwise, it is called \emph{disconnected}. 
An arbitrary matroid $M$ can be decomposed uniquely (up to ordering) as a direct sum $M=M_1\oplus\cdots\oplus M_k$, where $M_1,\ldots,M_k$ are connected matroids. In that case, the matroids $M_1,\ldots,M_k$ are called the \emph{connected components} of $M$.

Let $M$ be a connected matroid on $E$. Then $M$ is called \emph{minimally connected} if $M-e$ is disconnected for every $e\in E$. On the other hand, a series class $S$ of $M$ is said to be \emph{removable} if $M-S$ is connected. Evidently, every removable series class of a minimally connected matroid is non-trivial. For the existence of removable series classes we will need the following result.

\begin{lemma}\label{rsc}
 Let $M$ be a connected matroid on the ground set $E$ with at least two elements. Then $M$ contains a removable series class. In particular, if $M$ is minimally connected, then it contains a non-trivial removable series class.
 \end{lemma}
 
 \begin{proof}
If $M$ has exactly one series class, then $E$ forms a circuit and hence $E$ itself is a removable series class of $M$. When $M$ contains at least two series classes, the result follows from \cite[Proposition 5.3]{Sw04}. \end{proof}

Let $M_1$ and $M_2$ be matroids on ground sets $E_1$ and $E_2$ with $E_1\cap E_2=\{e\}$. Assume that $e$ is neither a loop nor a coloop of $M_1$ or $M_2$. Let $\mathcal{C}(M_i)$ denote the collection of circuits of $M_i$. The \emph{parallel connection} $P(M_1,M_2)$ of $M_1$ and $M_2$ with respect to $e$ is the matroid on $E_1\cup E_2$ whose collection of circuits is given by
\[
  \mathcal{C}(P(M_1,M_2))=\mathcal{C}(M_1)\cup\mathcal{C}(M_2)\cup\{C_1\cup C_2-e~:~ e\in C_i\in\mathcal{C}(M_i)\ \text{ for }\ i=1,2\}.
\]
The deletion $P(M_1,M_2)-e$ is called the \emph{$2$-sum} of $M_1$ and $M_2$, denoted by $M_1\oplus_2M_2$. Note that the circuits of $M_1\oplus_2M_2$ are the circuits of $P(M_1,M_2)$ not containing $e$; see \cite[3.1.14]{Ox11}. Thus
\begin{equation}\label{eq:2sum0}
 \mathcal{C}(M_1\oplus_2M_2)=\mathcal{C}(M_1-e)\cup\mathcal{C}(M_2-e)\cup\{C_1\cup C_2-e~:~ e\in C_i\in\mathcal{C}(M_i) \text{ for } i=1,2\}.
\end{equation}
 The following simple observation will be useful in \Cref{sec3}. For brevity's sake we call a matroid an \emph{$m$-circuit} if its ground set is an {$m$-circuit}.

\begin{lemma}\label{lm:2sum}
 Let $S$ be a series class of a matroid $M$ with $|S|=m$. Set $\widetilde{M}=M/(S-e)$ for some $e\in S$. Then $M\cong \widetilde{M}\oplus_2 C$, where $C$ is an $(m+1)$-circuit containing $e$.  
\end{lemma}

\begin{proof}
 By a slight abuse of notation we identify $C$ with its ground set. Then we may write $C=S'\cup e$, where $|S'|=|S|$. Notice that the collection $\mathcal{C}(\widetilde{M})$ of circuits of $\widetilde{M}$ consists of the minimal nonempty members of $\mathcal{D}:=\{D-(S-e): D\in \mathcal{C}(M)\}$; see \cite[Proposition 3.1.11]{Ox11}. Since $S$ is a series class, either $D\cap S=\emptyset$ or $S\subseteq D$ for every $D\in \mathcal{C}(M)$. Hence, all members of $\mathcal{D}$ are minimal and nonempty. This yields
$$\mathcal{C}(\widetilde{M})=\mathcal{D}=\{D: D\in \mathcal{C}(M),D\cap S=\emptyset\}\cup \{D-(S-e)~:~ D\in \mathcal{C}(M),S\subseteq D\}.$$Now by \eqref{eq:2sum0},
$$
  \begin{aligned}
   \mathcal{C}(\widetilde{M}\oplus_2 C)&=\mathcal{C}(\widetilde{M}-e)\cup \mathcal{C}(C-e)\cup\{C\cup D-S~:~ D\in \mathcal{C}(M),S\subseteq D\}\\
                                                          &=\{D~:~ D\in \mathcal{C}(M),D\cap S=\emptyset\}\cup \{S'\cup (D-S)~:~ D\in \mathcal{C}(M),S\subseteq D\}.
  \end{aligned}
$$
It then follows readily that $M\cong \widetilde{M}\oplus_2 C$, as desired.
\end{proof}

\begin{example}\label{ex}
	Let $M$ be the cycle matroid of the complete bipartite graph $K_{2,3}$, with the edges labelled as in \Cref{fig0}(a). Then $S=\{1,2\}$ is a series class of $M$. The 2-sum of $\widetilde{M}=M/\{1\}$ and the 3-circuit $C=\{2,1',2'\}$, which is the cycle matroid of the graph depicted in \Cref{fig0}(d), is clearly isomorphic to $M$. 
\end{example}

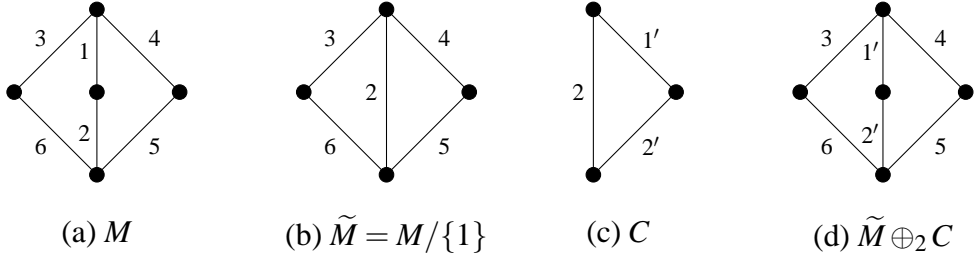
\begin{figure}[htb]
	
	\begin{tikzpicture}[scale=1.1]
	\draw  (1.,2.)-- (0.,1.);
	\draw  (0.,1.)-- (1.,0.);
	\draw  (1.,0.)-- (2.,1.);
	\draw  (2.,1.)-- (1.,2.);
	\draw  (1.,2.)-- (1.,0.);
	
	\draw  (4.5,2.)-- (5.5,1.);
	\draw  (5.5,1.)-- (4.5,0.);
	\draw  (4.5,0.)-- (3.5,1.);
	\draw  (3.5,1.)-- (4.5,2.);
	\draw  (4.5,2.)-- (4.5,0.);
	
	\draw  (7.,0.)-- (7.,2.);
	\draw  (7.,0.)-- (8.,1.);
	\draw  (8.,1.)-- (7.,2.);

	\draw  (10.5,2.)-- (9.5,1.);
	\draw  (9.5,1.)-- (10.5,0.);
	\draw  (10.5,0.)-- (11.5,1.);
	\draw  (11.5,1.)-- (10.5,2.);
	\draw  (10.5,2.)-- (10.5,0.);
	
	\draw (1.,-0.72) node {(a) $M$};
	\draw (4.5,-0.72) node {(b) $\widetilde{M}=M/\{1\}$};
	\draw (7.3,-0.72) node {(c) $C$};
	\draw (10.5,-0.72) node {(d) $\widetilde{M}\oplus_2 C$};
	
	\begin{scriptsize}
	\draw [fill=black] (1.,2.) circle (2.5pt);
	\draw [fill=black] (0.,1.) circle (2.5pt);
	\draw [fill=black] (1.,0.) circle (2.5pt);
	\draw [fill=black] (2.,1.) circle (2.5pt);
	\draw [fill=black] (1.,1.) circle (2.5pt);
	
	\draw (0.85,1.5) node {$1$};
	\draw (0.85,0.5) node {$2$};
	\draw (0.32,1.65) node {$3$};
	\draw (0.32,0.35) node {$6$};
	\draw (1.7,0.35) node {$5$};
	\draw (1.7,1.65) node {$4$};

	\draw [fill=black] (4.5,2.) circle (2.5pt);
	\draw [fill=black] (5.5,1.) circle (2.5pt);
	\draw [fill=black] (4.5,0.) circle (2.5pt);
	\draw [fill=black] (3.5,1.) circle (2.5pt);
	
	\draw (5.2,1.65) node {$4$};
	\draw (5.2,0.35) node {$5$};
	\draw (3.82,0.35) node {$6$};
	\draw (3.82,1.65) node {$3$};
	\draw (4.32,1.) node {$2$};

	\draw [fill=black] (7.,0.) circle (2.5pt);
	\draw [fill=black] (8.,1.) circle (2.5pt);
	\draw [fill=black] (7.,2.) circle (2.5pt);
	
	\draw (6.82,1) node {$2$};
	\draw (7.7,0.35) node {$2'$};
	\draw (7.7,1.65) node {$1'$};

	\draw [fill=black] (10.5,1.) circle (2.5pt);
	\draw [fill=black] (10.5,0.) circle (2.5pt);
	\draw [fill=black] (10.5,2.) circle (2.5pt);
	\draw [fill=black] (11.5,1.) circle (2.5pt);
	\draw [fill=black] (9.5,1.) circle (2.5pt);
	
	\draw (10.35,1.5) node {$1'$};
	\draw (10.35,0.5) node {$2'$};
	\draw (9.82,1.65) node {$3$};
	\draw (9.82,0.35) node {$6$};
	\draw (11.2,0.35) node {$5$};
	\draw (11.2,1.65) node {$4$};
		
	\end{scriptsize}
	\end{tikzpicture}
	\caption{$M\cong \widetilde{M}\oplus_2 C$}
	\label{fig0}
\end{figure}

By iterating, the operation of parallel connection can be defined for special families of more than two matroids. Let $M_1,\ldots,M_n$ be matroids on ground sets $E_1,\ldots,E_n$ such that $E_{i+1}\cap(\bigcup_{j=1}^i E_j)=\{e_i\}$ for $i=1,\ldots,n-1$. Here, $e_1,\ldots,e_{n-1}$ need not be distinct. Assume further that each $e_i$ is neither a loop nor a coloop of the matroids containing it. Then we can form $P(M_1,M_2)$, $P(P(M_1,M_2),M_3)$, and so on. The last matroid obtained in this way, denoted by $P(M_1,\ldots,M_n)$, is called the \emph{parallel connection} of $M_1,\ldots,M_n$ with respect to $e_1,\ldots,e_{n-1}$. 

Assume $M$ is a connected matroid on $E$. Then $M$ is called \emph{parallel irreducible at $e\in E$} if either $|E|=1$ or $M$ is not a parallel connection of two smaller matroids with respect to $e$. We say that $M$ is \emph{parallel irreducible} if it is parallel irreducible at every element of $E$. The following result, which was essentially proved by Brylawski \cite[Propositions 5.8, 5.9]{Bry71} (see also \cite[Lemma 2.1]{Le16}), indicates that in certain matroid arguments the general result can be obtained by restricting attention to the parallel irreducible case. 

\begin{lemma}\label{pr21}
 Let $M$ be a connected matroid on the ground set $E$. Then the following statements hold:
\begin{enumerate}
 \item 
If $M=P(M_1,M_2)$ with respect to $e$, then $M/e$ is disconnected: $$M/e=M_1/e\oplus M_2/e.$$ Conversely, if $M/e$ is disconnected, then $M$ is a parallel connection of two smaller matroids with respect to $e$. Hence, $M$ is parallel irreducible if and only if $M/e$ is connected for every $e\in E$.

\item 
$M$ admits a decomposition $M=P(M_1,\ldots,M_n)$, where each $M_i$ is connected and parallel irreducible. 
\end{enumerate}
\end{lemma}

\subsection{Broken circuit complexes}

Let $M$ be a matroid, whose ground set $E$ is endowed with a linear order $<$. We further assume that  $M$ is loopless, since otherwise $BC(M)=\emptyset$, which is not interesting for us here.
Let $r=r(M)$. Then it is well-known that $BC(M)$ is an $(r-1)$-dimensional shellable simplicial complex; see \cite{Pr77} or \cite[7.4]{Bj92}. Let $f(M)=(f_0(M),\ldots,f_r(M))$ be the $f$-vector of $BC(M)$, where $f_i(M)$ is the number of faces of $BC(M)$ of cardinality $i$. Notice that $f(M)$ is independent of the chosen order $<$, as is easily seen from the Whitney--Rota formula \eqref{WR}.
Define the \emph{$h$-vector} $h(M)=(h_0(M),\ldots,h_r(M))$ and the \emph{$h$-polynomial} (or \emph{shelling polynomial}) $h(M;t)=\sum_{i=0}^rh_i(M)t^{r-i}$ of $BC(M)$ by the polynomial identity $h(M;t)=(-1)^r\chi(M;1-t)$. Thus, the $f$-vector and the $h$-vector of $BC(M)$ are correlated as follows
\[f_i(M)=\sum_{j=0}^i\binom{r-j}{i-j}h_j(M) \ \ \text{and}\ \
 h_i(M)=\sum_{j=0}^i(-1)^{i-j}\binom{r-j}{i-j}f_j(M),\ \ i=0,\ldots,r.\]
In the sequel, for convenience, we make the convention that $h_i(M)=0$ for $i<0$ or $i>r$. Moreover, when it is clear from the context which matroid we are referring to, we will just write $h_i$ instead of $h_i(M)$. 

Note that both $\chi(M;t)$ and $h(M;t)$ are, up to sign, evaluations of the \emph{Tutte polynomial} $T(M;x,y)$ of $M$, which is defined by
\[T(M;x,y)=\sum_{X\subseteq E}(x-1)^{r(E)-r(X)}(y-1)^{|X|-r(X)}.\]
Evidently, $\chi(M;t)=(-1)^rT(M;1-t,0)$. Hence, $h(M;t)=T(M;t,0)$. 

 For later usage we collect here several basis properties of the $h$-polynomial of $BC(M)$. They follow easily from the corresponding properties of the Tutte polynomial of $M$; see \cite[6.2]{BO92} and \cite[p. 182]{Bry82}.

\begin{lemma}\label{hpol}
 Let $M$ be a loopless matroid of rank $r$ on the ground set $E$. Let $h(M;t)=\sum_{i=0}^rh_it^{r-i}$ be the $h$-polynomial of $BC(M)$. Then the following statements hold:
\begin{enumerate}
 \item
$h_i\geq 0$ for $i=0,\ldots,r$. Moreover, if $M$ has $c$ connected components, then $r-c$ is the largest index $i$ such that $h_{i}\ne0$. 

\item (Deletion-contraction)
Suppose $|E|\geq 2$ and $e\in E$. Then
\[h(M;t)=
\begin{cases}
 th(M-e;t)  &\text{ if } e \text{ is a coloop of } M,\\
h(M-e;t)+h(M/e;t) & \text{ otherwise}.
\end{cases}
\]
Thus, in particular, if $M$ is connected, then either $M-e$ or $M/e$ is connected.

\item
If $M$ is an $(r+1)$-circuit, then $h(M;t)=t^r+t^{r-1}+\cdots+t$ .

\item
Assume that $M$ is either the direct sum or the parallel connection of two matroids $M_1$ and $M_2$. Then
\[h(M;t)=
\begin{cases}
 h(M_1;t)h(M_2;t)& \text{ if } M=M_1\oplus M_2,\\
 t^{-1}h(M_1;t)h(M_2;t)& \text{ if } M=P(M_1,M_2).
\end{cases}
\]

 \end{enumerate}
\end{lemma}

As an important step in the proof of \Cref{th32}, we will relate the $h$-vector of $BC(M)$ to the $h$-vectors of broken circuit complexes of certain minors of $M$ which are obtained from $M$ by deleting or contracting elements in a series class. For this, the following simple facts will be necessary.

\begin{lemma}\label{lm21}
 Let $S=\{e_1,\ldots,e_m\}$ be a series class of a loopless matroid $M$. For $0\le j\le m-1$, set $M_j=M/\{e_1,\ldots,e_j\}$ and $S_j=\{e_{j+1},\ldots,e_m\}$. Then the following statements hold:
\begin{enumerate}
 \item $r(M-S)=r(M)-m+1$.

\item $r(M_j)=r(M)-j$, and if $M$ is connected, so is $M_j$.

 \item $S_j$ is a series class of $M_j$ and $M_j-S_j= M-S$.

\item For every $e\in S$ and $e'\in S_j$ the $h$-vectors of the broken circuit complexes of the matroids $M-e,\ M-S$ and $M_j-e'$ coincide.
\end{enumerate}

\end{lemma}

\begin{proof}
 (i) Since, by definition, $e_1$ is not a coloop of $M$, we have that $r(M-e_1)=r(M)$; see, e.g., \cite[3.1.5]{Ox11}. Now as every element of $S_1$ is a coloop of $M-e_1$, it holds that
$$r(M-S)=r((M-e_1)-S_1)=r(M-e_1)-|S_1|=r(M)-m+1.$$

(ii) As $M_j=M_{j-1}/e_j$ and $e_j$ is not a loop of $M_{j-1}$, we have $r(M_j)=r(M_{j-1})-1$; see, e.g., \cite[3.1.7]{Ox11}. In addition, $M_{j-1}-e_j$ is not connected since every element of $S_j$ is a coloop of this matroid. Hence, by \Cref{hpol}(ii), $M_j$ is connected if $M_{j-1}$ is so. The assertion now follows by induction.

(iii) By definition, it is easy to see that $S_j$ is a series class of $M_j$. Now since $e_1,\ldots,e_j$ are coloops of $M-S_j$, it follows from \cite[Corollary 3.1.25]{Ox11} that
$$M_j-S_j=(M-S_j)/\{e_1,\ldots,e_j\}=(M-S_j)-\{e_1,\ldots,e_j\} =M-S.$$

(iv) Since the elements of $S-e$ are coloops of $M-e$, \Cref{hpol}(ii) yields $h(M-e;t)=t^{m-1}h(M-S;t)$. Similarly, $h(M_j-e';t)=t^{m-j-1}h(M_j-S_j;t)$. As $M-S=M_j-S_j$ by (iii), the assertion follows.
\end{proof}

\section{Flawlessness of $h$-vectors of broken circuit complexes}\label{sec3}

This section is devoted to the proof of \Cref{th32} and its applications. We begin with the following lemma, which is essential for reducing the proof of \Cref{th32} to the case of minimally connected matroids. Recall that a sequence $(a_0,a_1,\ldots,a_n)$ is \emph{symmetric} if $a_i=a_{n-i}$ for $0\le i\le n$. Let us say that a polynomial $a_0t^{n+u}+a_1t^{n+u-1}+\cdots +a_nt^u$ with $a_0, a_n\ne 0$ and $u\ge0$ has a certain property (such as symmetric, unimodal or strongly {flawless}) if its coefficient sequence $(a_0,a_1,\ldots,a_n)$ has that property. 

\begin{lemma}\label{red}
 If $\varphi(t)$ and $\psi(t)$ are strongly flawless polynomials with nonnegative coefficients, then so is their product.
\end{lemma}

\begin{proof}
By definition, a polynomial is strongly flawless if and only if its product with any power $t^u$ ($u\ge0$) is so. Hence without loss of generality we may assume that $\varphi(t)$ and $\psi(t)$ have the following form:
$$\begin{aligned}
   \varphi(t)&=a_0t^{n}+a_1t^{n-1}+\cdots +a_{n-1}t+a_n,\\
   \psi(t)&=b_0t^{m}+b_1t^{m-1}+\cdots +b_{m-1}t+b_m,
  \end{aligned}
$$
where $a_0,a_n,b_0,b_m> 0$. We will argue by induction on
$$d_{\varphi, \psi}:=|\{0\le i\le \lfloor n/2\rfloor: a_i<a_{n-i}\}|+|\{0\le j\le \lfloor m/2\rfloor: b_j<b_{m-j}\}|.$$
If $d_{\varphi, \psi}=0$, then $\varphi(t)$ and $\psi(t)$ are symmetric polynomials. 
Observe that for a symmetric polynomial, being strongly flawless is equivalent to being unimodal. So $\varphi(t)$ and $\psi(t)$ are symmetric and unimodal. It follows that their product $\varphi(t)\psi(t)$ is also symmetric and unimodal (see, e.g., \cite[Proposition 1]{Sta89}). Thus, $\varphi(t)\psi(t)$ is  strongly flawless, and we are done in this case.

Now consider the case $d_{\varphi, \psi}>0$. We may suppose that $a_i<a_{n-i}$ for some $0\le i\le \lfloor n/2\rfloor$. Set $k:=\min\{0\le i\le \lfloor n/2\rfloor: a_i<a_{n-i}\}$. Let $\overline{\varphi}(t)$ be the polynomial obtained from $\varphi(t)$ by replacing the term $a_{n-k}t^k$ of $\varphi(t)$ with $a_{k}t^k$, i.e.,
$\overline{\varphi}(t)=\varphi(t)+(a_k-a_{n-k})t^{k}.$
Then it is readily seen that $\overline{\varphi}(t)$ is strongly flawless. Moreover, $d_{\overline{\varphi}, \psi}=d_{\varphi, \psi}-1$.  Writing $\varphi(t)\psi(t)=\sum_{i=0}^{m+n}c_it^{m+n-i}$ and $\overline{\varphi}(t)\psi(t)=\sum_{i=0}^{m+n}c'_it^{m+n-i}$, we get
\begin{equation}\label{eq:red}
 c_{i}=\sum_{u+v=i}a_ub_v
=\begin{cases}
  c'_i& \text{if }\ i< n-k\ \text{or }\ i> m+n-k,\\
c'_i+(a_{n-k}-a_k)b_{i+k-n}& \text{otherwise}.
 \end{cases}
\end{equation}
Since $a_{n-k}>a_k$ and the coefficients of $\psi(t)$ are nonnegative, it holds that $c_i\ge c'_i$ for all $i$. Now
let $0\le i\le j\le m+n-i$. We have to show that $c_i\le c_j$. Note that $c'_i\le c'_j$ by induction. So, if $i<n-k$, then it follows from \eqref{eq:red} that $c_i=c'_i\le c'_j\le c_j$. 
Now suppose $i\ge n-k$. Then $i\ge k$ since $k\le \lfloor n/2\rfloor$. Hence $j\le m+n-i \le m+n-k$. Again by \eqref{eq:red} we have
$$c_j-c_i=c'_j-c'_i+(a_{n-k}-a_k)(b_{j+k-n}-b_{i+k-n})\ge (a_{n-k}-a_k)(b_{j+k-n}-b_{i+k-n}).$$
Thus, the inequality $c_i\le c_j$ will be confirmed once we have shown that $b_{i+k-n}\le b_{j+k-n}$. But the last inequality holds since $0\le i+k-n\le j+k-n\le m-(i+k-n)$ (which follows easily from $n-k\le i\le j\le m+n-i$ and $k\le \lfloor n/2\rfloor$) and $\psi(t)$ is strongly flawless. This completes the proof.
\end{proof}

In the sequel, for our purposes, it will be convenient to consider $h$-vectors with zero entries at the end removed. So, if we say that $h(M)=(h_0(M),h_1(M),\ldots,h_s(M))$ is the $h$-vector of $BC(M)$, then $s$ is the largest index $i$ with $h_i(M)\ne 0$. In this case, recall from \Cref{hpol}(i) that $s=r-c$, where $r=r(M)$ and $c$ is the number of connected components of $M$. 

Now let $M$ be a loopless matroid and let $h(M)=(h_0(M),h_1(M),\ldots,h_s(M))$ be the $h$-vector of $BC(M)$. Define 
$$\bar{h}_i(M):=
\begin{cases}
 h_{s-i}(M)-h_i(M) &\text{for }\ 0\le i\le \lfloor s/2\rfloor,\\
0& \text{otherwise}.
\end{cases}
$$ 
Following Swartz \cite{Sw06}, we call $\bar h(M):=(\bar h_0(M),\bar h_1(M),\ldots,\bar h_{\lfloor s/2\rfloor}(M))$ the \emph{complementary $h$-vector} of $BC(M)$. For convenience we set $\bar h(M)=(0)$ if $M$ contains a loop.

The next two lemmas present two different interpretations of the complementary $h$-vector of $BC(M)$ which involve the $h$-vector of $BC(M/S)$, where $S$ is a (removable) series class of $M$. Recall our convention that $h_i(M/S)=0$ for $i<0$ or $i>r(M/S)$.

\begin{lemma}\label{series1}
 Let $M$ be a connected matroid and $S$ a non-trivial removable series class of $M$ with $|S|=m$. Let $h(M)=(h_0(M),h_1(M),\ldots,h_s(M))$ be the $h$-vector of  $BC(M)$. Then for every $e\in S$ and $0\le i\le \lfloor s/2\rfloor$,
$$\bar h_i(M)=\bar h_{i}(M/e)+\bar h_{i-m+1}(M-S)+(h_{i-m+1}(M/S)-h_{i-m}(M/S)).$$
\end{lemma}

\begin{proof}
If $M$ is a $2$-circuit, then the statement is easily seen to be true. So assume that $M$ is not a $2$-circuit. Suppose $S=\{e_1,\ldots,e_m\}$ with $e=e_1$. Set $M_j=M/\{e_1,\ldots,e_j\}$ for $j=1,\ldots,m$. We will show via induction that
\begin{align}
\begin{aligned}
  \bar h_i(M)  = &\ \bar h_{i}(M/e_1)+\bar{h}_{i-m+1}(M-S)+(h_{i-j+1}(M_j)-h_{i-j}(M_j))\\
                        &+(h_{i-m+1}(M-S)-h_{i-j+1}(M-S))
\end{aligned}
\label{eq31}
 \end{align}
for $j=1,\ldots,m$. The case $j=m$ then gives the desired assertion.

Using the deletion-contraction formula (\Cref{hpol}(ii)) and \Cref{lm21}(iv), we have
\begin{align}
	\begin{aligned}
   \bar h_i(M)  =&\ h_{s-i}(M)-h_i(M)\\
                        = &\ (h_{s-i}(M-e_1)+h_{s-i-1}(M/e_1))-(h_{i}(M-e_1)+h_{i-1}(M/e_1))\\
                        =&\ (h_{s-i-1}(M/e_1)-h_{i}(M/e_1))+(h_{s-i}(M-S)-h_{i-m+1}(M-S))\\
                        &+(h_{i}(M/e_1)-h_{i-1}(M/e_1))+(h_{i-m+1}(M-S)-h_{i}(M-S)).
  \end{aligned}
\label{eq32}
\end{align}
By \Cref{lm21}(ii), $M/e_1$ is connected and $r(M/e_1)=r(M)-1$. Thus, in particular, $M/e_1$ is loopless since $M$ is not a 2-circuit. So from \Cref{hpol}(i) it follows that $h_{s-1}(M/e_1)\ne 0$, and hence $\bar h_i(M/e_1)=h_{s-i-1}(M/e_1)-h_{i}(M/e_1)$. Similarly, as $M-S$ is connected and $r(M-S)=r(M)-m+1$ (see \Cref{lm21}(i)), it holds that $h_{s-m+1}(M-S)\ne 0$ and $\bar h_{i-m+1}(M-S)=h_{s-i}(M-S)-h_{i-m+1}(M-S)$. Thus \eqref{eq32} implies that \eqref{eq31} is true for $j=1$. To complete the induction argument, it suffices to show that 
$$\begin{aligned}
  &(h_{i-j+1}(M_j)-h_{i-j}(M_j))+(h_{i-m+1}(M-S)-h_{i-j+1}(M-S))\\
=&(h_{i-j}(M_{j+1})-h_{i-j-1}(M_{j+1}))+(h_{i-m+1}(M-S)-h_{i-j}(M-S)),
\end{aligned}$$
or equivalently,
$$\begin{aligned}
  h_{i-j+1}(M_j)-h_{i-j}(M_j)=&\ (h_{i-j}(M_{j+1})+h_{i-j+1}(M-S))\\
                                              &-(h_{i-j-1}(M_{j+1})+h_{i-j}(M-S)).
\end{aligned}$$
But the last equality follows from the deletion-contraction formula, since $M_{j+1}=M_j/e_{j+1}$ and $h_k(M-S)=h_k(M_j-e_{j+1})$ (by \Cref{lm21}(iv)). This finishes the proof.
\end{proof}

\begin{lemma}\label{series2}
 Let $M$ be a connected matroid and $S$ a series class of $M$ with $|S|=m$. Set $\widetilde{M}=M/(S-e)$ for some $e\in S$. Let $h(M)=(h_0(M),h_1(M),\ldots,h_s(M))$ be the $h$-vector of $BC(M)$. Then 
$$
   \bar h_i(M)= \begin{dcases}
                              \sum_{j=0}^{\min\{i,{s-m-i}\}}&\bar h_{j}(\widetilde{M})+\sum_{j=1}^{i}
                              (h_{i-j}(M/S)-h_{s-m+1-j}(M/S))\\
                          &\text{ if }\ 0\le i\le\min\{m-1,s-m+1\},\\
                               \sum_{j=i-m+1}^{\min\{i,{s-m-i}\}}&\bar h_{j}(\widetilde{M})+\sum_{j=1}^{m-1} (h_{i-j}(M/S)-h_{s-i-j}(M/S))\\
                          &\text{ if }\ m-1\le s-m+1\ \text{ and }\ m-1\le i\le\lfloor s/2\rfloor,\\
                           0 &\text{ if }\ s-m+1\le m-1\ \text{ and }\ s-m+1\le i\le\lfloor s/2\rfloor.
                            \end{dcases}
$$
\end{lemma}

\begin{proof}
Note that $\widetilde{M}$ is connected by \Cref{lm21}(ii). So $\widetilde{M}$ contains a loop if and only if it is itself a loop, which means that $M$ is a circuit. Since the lemma is clearly true in this case, we may henceforth assume that $\widetilde{M}$ is loopless. By \Cref{lm:2sum}, $M\cong P(\widetilde{M}, C)-e$, where $C$ is an $(m+1)$-circuit containing $e$. Thus, the deletion-contraction formula, \Cref{pr21}(i) and \Cref{hpol}(iv) yield
\begin{equation}
 \begin{aligned}
   h(M;t)&=h(P(\widetilde{M}, C);t)-h(P(\widetilde{M}, C)/e;t)\\
              &=h(P(\widetilde{M}, C);t)-h(\widetilde{M}/e\oplus C/e;t)\\
              &=\frac{h(\widetilde{M};t)h(C;t)}{t}-h(M/S;t)h(C/e;t).
  \end{aligned}
\label{eq:2sum}
\end{equation}
Since $r(\widetilde{M})=r(M)-m+1=s-m+2$ (see \Cref{lm21}(ii)) and $r(M/S)=r(\widetilde{M})-1=s-m+1$, we may write
$$\begin{aligned}
   h(\widetilde{M};t)&=h_0(\widetilde{M})t^{s-m+2}+h_1(\widetilde{M})t^{s-m+1}+\cdots+h_{s-m+1}(\widetilde{M})t \quad\mbox{   and }\\
   h(M/S;t)&=h_0(M/S)t^{s-m+1}+h_1(M/S)t^{s-m}+\cdots+h_{s-m}(M/S)t.
  \end{aligned}
$$
Plugging these polynomials into \eqref{eq:2sum} and using \Cref{hpol}(iii) we get
\begin{equation}
 h(M;t)=\left(\sum_{j=0}^{s-m+1}h_j(\widetilde{M})t^{s-m+2-j}\right)\left(\sum_{k=0}^{m-1}t^k\right)-\left(\sum_{j=0}^{s-m}h_j(M/S)t^{s-m+1-j}\right)\left(\sum_{k=1}^{m-1}t^k\right).
\label{eq:2sum2}
\end{equation}
From this formula we will derive formulas for the coefficients of $h(M;t)$, and thereby obtain the desired formula for the complementary $h$-vector. We distinguish two cases:\medskip

 \noindent{\sf Case 1:} $m-1\le s-m+1$.\\
Note that $h_i(M)$ is the coefficient of $t^{s-i+1}$ in $h(M;t)$. So from \eqref{eq:2sum2} we get
\begin{equation}
h_i(M)=\begin{dcases}
               \sum_{j=0}^ih_j(\widetilde{M}) -   \sum_{j=0}^{i-1}h_j(M/S)& \text{ for }\ i\le m-1,\\
               \sum_{j=i-m+1}^ih_j(\widetilde{M}) -   \sum_{j=i-m+1}^{i-1}h_j(M/S)& \text{ for }\ m-1\le i\le s-m+1,\\
               \sum_{j=i-m+1}^{s-m+1}h_j(\widetilde{M}) -   \sum_{j=i-m+1}^{s-m}h_j(M/S)& \text{ for }\ s- m+1\le i\le s.
             \end{dcases}
\label{eq:2sum3}
\end{equation}
As $\widetilde{M}$ is loopless and connected, it follows from \Cref{hpol}(i) that $h_{s-m+1}(\widetilde{M})\ne 0$. Thus $\bar h_{j}(\widetilde{M})=h_{s-m+1-j}(\widetilde{M})-h_{j}(\widetilde{M})$ for $0\le j\le \lfloor \frac{s-m+1}{2}\rfloor$. Now it is readily seen from \eqref{eq:2sum3} that
$$
   \bar h_i(M)=h_{s-i}(M)-h_i(M)
                      = \begin{dcases}
                              \sum_{j=0}^{\min\{i,{s-m-i}\}}&\bar h_{j}(\widetilde{M})+\sum_{j=1}^{i} (h_{i-j}(M/S)-h_{s-m+1-j}(M/S))\\
                          &\text{ for }\ 0\le i\le m-1,\\
                               \sum_{j=i-m+1}^{\min\{i,{s-m-i}\}}&\bar h_{j}(\widetilde{M})+\sum_{j=1}^{m-1} (h_{i-j}(M/S)-h_{s-i-j}(M/S))\\
                          &\text{ for } \ m-1\le i\le\lfloor s/2\rfloor.
                            \end{dcases}
$$

\noindent{\sf Case 2:} $s-m+1< m-1$.\\
In this case, \eqref{eq:2sum2} gives
$$h_i(M)=\begin{dcases}
               \sum_{j=0}^ih_j(\widetilde{M}) -   \sum_{j=0}^{i-1}h_j(M/S)& \text{ for }\ i\le s-m+1,\\
               \sum_{j=0}^{s-m+1}h_j(\widetilde{M}) -   \sum_{j=0}^{s-m}h_j(M/S)& \text{ for }\ s-m+1\le i\le m-1,\\
               \sum_{j=i-m+1}^{s-m+1}h_j(\widetilde{M}) -   \sum_{j=i-m+1}^{s-m}h_j(M/S)& \text{ for }\ m-1\le i\le s.
             \end{dcases}$$
Hence
$$
   \bar h_i(M)  = \begin{dcases}
                              \sum_{j=0}^{\min\{i,{s-m-i}\}}&\bar h_{j}(\widetilde{M})+\sum_{j=1}^{i} (h_{i-j}(M/S)-h_{s-m+1-j}(M/S))\\
                          &\text{ for }\ 0\le i\le s-m+1,\\
                               0 &\text{ for } \ s-m+1\le i\le\lfloor s/2\rfloor.
                            \end{dcases}
$$

The desired formula for $\bar h_i(M)$ is obtained by combining the two cases above.
\end{proof}

\begin{example}
Let us revisit the cycle matroid $M$ of the complete bipartite graph $K_{2,3}$ discussed in \Cref{ex}. Notice that the series class $S=\{1,2\}$ of $M$ is removable. The graphs corresponding to the minors $\widetilde{M}=M/\{1\},\ M-S,\ M/S$ of $M$ are depicted in \Cref{fig1}. Using \Cref{hpol} one easily finds that $h(M/S;t)=t^2,$ $h(M-S;t)=t^3+t^2+t,$ $h(\widetilde{M};t)=t^3+2t^2+t$, and $h(M;t)=t^4+2t^3+3t^2+t$. Thus $\bar h(M)=(1-1,3-2)=(0,1)$. This agrees with the computation of $\bar h(M)$ using \Cref{series1} or \Cref{series2}. For example, by \Cref{series1},
$$\bar h_1(M)=\bar h_1(M/\{1\})+\bar h_0(M-S)+(h_0(M/S)-h_{-1}(M/S))=0+0+(1-0)=1.$$
On the other hand, by \Cref{series2},
$$\bar h_1(M)=\bar h_0(\widetilde{M})+(h_0(M/S)-h_{1}(M/S))=0+(1-0)=1.$$
\end{example}

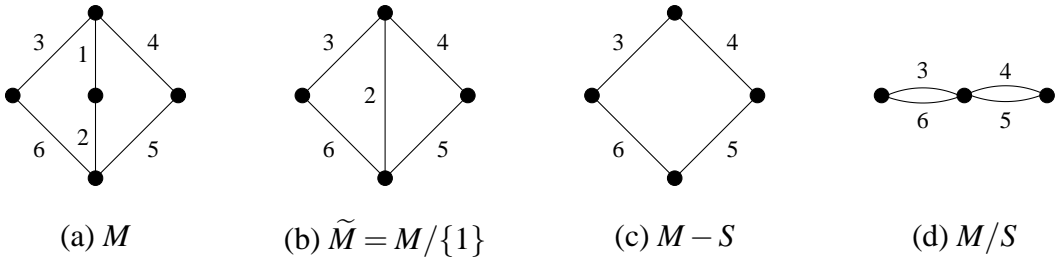
\begin{figure}[htb]

	\begin{tikzpicture}[scale=1.1]
	\draw  (1.,2.)-- (0.,1.);
	\draw  (0.,1.)-- (1.,0.);
	\draw  (1.,0.)-- (2.,1.);
	\draw  (2.,1.)-- (1.,2.);
	\draw  (1.,2.)-- (1.,0.);
	
	\draw  (4.5,2.)-- (5.5,1.);
	\draw  (5.5,1.)-- (4.5,0.);
	\draw  (4.5,0.)-- (3.5,1.);
	\draw  (3.5,1.)-- (4.5,2.);
	\draw  (4.5,2.)-- (4.5,0.);
	
	\draw  (7.,1.)-- (8.,0.);
	\draw  (8.,0.)-- (9.,1.);
	\draw  (9.,1.)-- (8.,2.);
	\draw  (8.,2.)-- (7.,1.);
	
	\draw [shift={(11,0.)}] plot[domain=1.156289452210111:2.081284648711673,variable=\t]({1.*1.0925200226998129*cos(\t r)+0.*1.0925200226998129*sin(\t r)},{0.*1.0925200226998129*cos(\t r)+1.*1.0925200226998129*sin(\t r)});
	\draw [shift={(11,2.22)}] plot[domain=4.337548265029092:5.115300571238205,variable=\t]({1.*1.311030129325791*cos(\t r)+0.*1.311030129325791*sin(\t r)},{0.*1.311030129325791*cos(\t r)+1.*1.311030129325791*sin(\t r)});
	\draw [shift={(12,0.)}] plot[domain=1.1071487177940904:2.0344439357957027,variable=\t]({1.*1.118033988749895*cos(\t r)+0.*1.118033988749895*sin(\t r)},{0.*1.118033988749895*cos(\t r)+1.*1.118033988749895*sin(\t r)});
	\draw [shift={(12,2.16)}] plot[domain=4.365003017466369:5.176036589385496,variable=\t]({1.*1.233693641063291*cos(\t r)+0.*1.233693641063291*sin(\t r)},{0.*1.233693641063291*cos(\t r)+1.*1.233693641063291*sin(\t r)});
	
	\draw (1.,-0.72) node {(a) $M$};
	\draw (4.5,-0.72) node {(b) $\widetilde{M}=M/\{1\}$};
	\draw (8.,-0.72) node {(c) $M-S$};
	\draw (11.5,-0.72) node {(d) $M/S$};
	
	\begin{scriptsize}
	\draw [fill=black] (1.,2.) circle (2.5pt);
	\draw [fill=black] (0.,1.) circle (2.5pt);
	\draw [fill=black] (1.,0.) circle (2.5pt);
	\draw [fill=black] (2.,1.) circle (2.5pt);
	\draw [fill=black] (1.,1.) circle (2.5pt);
	
    \draw (0.85,1.5) node {$1$};
    \draw (0.85,0.5) node {$2$};
	\draw (0.32,1.65) node {$3$};
	\draw (0.32,0.35) node {$6$};
	\draw (1.7,0.35) node {$5$};
	\draw (1.7,1.65) node {$4$};

	\draw [fill=black] (4.5,2.) circle (2.5pt);
	\draw [fill=black] (5.5,1.) circle (2.5pt);
	\draw [fill=black] (4.5,0.) circle (2.5pt);
	\draw [fill=black] (3.5,1.) circle (2.5pt);
	
	\draw (5.2,1.65) node {$4$};
	\draw (5.2,0.35) node {$5$};
	\draw (3.82,0.35) node {$6$};
	\draw (3.82,1.65) node {$3$};
	\draw (4.32,1.) node {$2$};

	\draw [fill=black] (7.,1.) circle (2.5pt);
	\draw [fill=black] (8.,0.) circle (2.5pt);
	\draw [fill=black] (9.,1.) circle (2.5pt);
	\draw [fill=black] (8.,2.) circle (2.5pt);

	\draw (7.32,0.35) node {$6$};
	\draw (8.7,0.35) node {$5$};
	\draw (8.7,1.65) node {$4$};
	\draw (7.32,1.65) node {$3$};

	\draw [fill=black] (10.5,1.) circle (2.5pt);
	\draw [fill=black] (11.5,1.) circle (2.5pt);
	\draw [fill=black] (12.5,1.) circle (2.5pt);
	
	\draw (11.,1.3) node {$3$};
	\draw (11.,0.7) node {$6$};
	\draw (12.,1.3) node {$4$};
	\draw (12.,0.7) node {$5$};
	\end{scriptsize}
	\end{tikzpicture}
	\caption{Minors of $M$ related to removable series class $S=\{1,2\}$}\label{fig1}
\end{figure}

We are now ready to prove our main result.

\begin{proof}[Proof of \Cref{th32}] 
Let $M\in\mathcal{M}$ and let $h(M)=(h_0(M),h_1(M),\ldots,h_s(M))$ be the $h$-vector of $BC(M)$. Since $h(M)$ is unimodal by assumption, it suffices to prove that $h(M)$ is flawless, i.e., the complementary $h$-vector of $BC(M)$ is nonnegative. We proceed by induction on the cardinality of the ground set $E$ of $M$. 

If $|E|=1$, then $h(M)=(1)$ and we have nothing to prove. So suppose $|E|\ge2$. We first show that we can reduce to the case where $M$ is minimally connected. By Lemmas \ref{pr21}(ii), \ref{hpol}(iv) and \ref{red}, we may assume that $M$ is connected, and furthermore, parallel irreducible. Thus, by \Cref{pr21}(i), $M/e$ is connected for every $e\in E$. We will show that $\bar h_i(M)\ge 0$ for $0\le i\le \lfloor s/2\rfloor$ if there exists $e\in E$ with $M-e$ connected. Indeed, if $s$ is even and $i=s/2$, then $\bar h_i(M)=0$. Now assume that $s$ is odd or $i<s/2$. Then $i\le \lfloor (s-1)/2\rfloor$. Using the deletion-contraction formula we have
\begin{align}
 \begin{aligned}
   \bar h_i(M)=&\ h_{s-i}(M)-h_i(M)\\
                    =&\ (h_{s-i}(M-e)+h_{s-i-1}(M/e))-(h_{i}(M-e)+h_{i-1}(M/e))\\
                    =&\ (h_{s-i}(M-e)-h_{i}(M-e))+(h_{s-i-1}(M/e)-h_{i}(M/e))\\
                      &+(h_{i}(M/e)-h_{i-1}(M/e))\\
                    =&\ \bar h_i(M-e)+\bar h_i(M/e)+(h_{i}(M/e)-h_{i-1}(M/e)).
  \end{aligned}
  \label{eq:mc}
\end{align}
 The last equality follows since $M-e$ and $M/e$ are connected. By the induction hypothesis, the $h$-vectors of $BC(M-e)$ and $BC(M/e)$ are strongly flawless, implying that each summand of $\bar h_i(M)$ in the last row of  \eqref{eq:mc} is nonnegative. Therefore, $\bar h_i(M)$ is nonnegative as well.

Henceforth we may assume that $M$ is minimally connected. Then $M$ contains a non-trivial removable series class by \Cref{rsc}. Let $S$ be such a series class of $M$ with $|S|=m$. Given $0\le i\le \lfloor s/2\rfloor$, let us verify that $\bar h_i(M)\ge 0$. If $i\le m-1$, then $\bar h_i(M)\ge 0$ by \Cref{series1} and the induction hypothesis. Now consider the case $i>m-1$. Since $i\le \lfloor s/2\rfloor$, we must have $m-1<s-m+1$. It then follows from Lemmas \ref{series1}, \ref{series2} and the induction hypothesis that
$$\bar h_i(M)\ge \max\{h_{i-m+1}(M/S)-h_{i-m}(M/S),\sum_{j=1}^{m-1} (h_{i-j}(M/S)-h_{s-i-j}(M/S))\}.$$
Thus, if $h_{i-m+1}(M/S)\ge h_{i-m}(M/S)$, then $\bar h_i(M)\ge 0$. Suppose now that $h_{i-m}(M/S)>h_{i-m+1}(M/S)$. Then the unimodality of the $h$-vector of $BC(M/S)$ yields
$$
 h_{i-m+1}(M/S)\ge \cdots\ge h_{i-1}(M/S)\ge \cdots \ge h_{s-i-1}(M/S).
$$
It follows that for $1\le j \le m-1$, we have $h_{i-j}(M/S)\ge  h_{s-i-j}(M/S)$, because $i-m+1\le i-j\le s-i-j\le s-i-1$. Hence
$$\sum_{j=1}^{m-1} (h_{i-j}(M/S)-h_{s-i-j}(M/S))\ge 0,$$
which also implies that $\bar h_i(M)\ge 0$. The proof is complete.
\end{proof}

As a consequence of \Cref{th32} we verify \Cref{con} for matroids representable over a field of characteristic zero.

\begin{corollary}\label{co34}
 Let $M$ be a matroid representable over a field of characteristic zero. Then the $h$-vector of $BC(M)$ is strongly flawless. 
\end{corollary}

\begin{proof}
 Let $\mathcal{M}$ be the class of matroids representable over a field of characteristic zero. Then it is well-known that $\mathcal{M}$ is minor-closed; see \cite[Proposition 3.2.4]{Ox11}. Moreover, it follows from Huh's log-concavity result \cite[Theorem 3]{Hu15} that $\mathcal{M}$ is unimodal. So $\mathcal{M}$ is strongly flawless by \Cref{th32}.
\end{proof}

Let us now derive an application of \Cref{co34} to Orlik--Terao algebras. Recall that a (central) \emph{complex hyperplane arrangement} $\mathcal{A}=\{H_1,\ldots,H_n\}$ is a collection of hyperplanes in $\mathbb{C}^{r}$, all of which contain the origin of $\mathbb{C}^{r}$. Suppose each hyperplane $H_i$ of $\mathcal{A}$ is given as the kernel of a linear form $\alpha_i$. Then the \emph{Orlik--Terao algebra} of $\mathcal{A}$ is defined to be the $\mathbb{C}$-algebra generated by reciprocals of the $\alpha_i$'s:
$$C(\mathcal{A}):=\mathbb{C}[1/\alpha_1,\ldots,1/\alpha_n].$$
This algebra was introduced by Orlik and Terao in \cite{OT94}. Since then it has appeared in different contexts and received considerable attention; see e.g., \cite{ BP09, DGT14, Le14, LM15, LR13, MP15, PS06, SSV13, Sc11, ST09, Te02}. An interesting property of $C(\mathcal{A})$ is that it degenerates flatly to the Stanley--Reisner ring of the broken circuit complex of the \emph{underlying matroid} $M(\mathcal{A})$ of $\mathcal{A}$ \cite[Theorem 4]{PS06}. Thus, in particular, $C(\mathcal{A})$ is a Cohen--Macaulay ring and its $h$-vector coincides with the $h$-vector of $BC(M(\mathcal{A}))$.
Recall that the underlying matroid $M(\mathcal{A})$ is defined to be the matroid on ground set $\mathcal{A}$ such that a subset $B = \{H_{ i_ 1},\ldots , H _{i_ p}\}$ of $\mathcal{A}$ is independent if and only if the corresponding linear forms $\alpha_{i_1}, \ldots,\alpha_{i_p}$ are linearly independent. Evidently, $M(\mathcal{A})$ is representable over $\mathbb{C}$. So from \Cref{co34} we immediately get the following:

\begin{corollary}
 Let $\mathcal{A}$ be a complex hyperplane arrangement. Then the $h$-vector of the Orlik--Terao algebra of  $\mathcal{A}$ is strongly flawless.
\end{corollary}

It should be noted here that $C(\mathcal{A})$ has a canonical linear system of parameters \cite[Proposition 7]{PS06} and that, similar to Swartz's examples mentioned in the introduction, the corresponding Artinian reduction of $C(\mathcal{A})$ needs not have $g$-elements \cite[Remark 8]{PS06}. It would therefore be difficult to provide an algebraic proof of the above corollary.

\section{Concluding remarks}\label{sec4}

In view of our main result (\Cref{th32}), \Cref{con} would follow from the first one of the following successively stronger conjectured assertions:

\begin{conjecture}\label{con2}
Let $h(M)=(h_0,h_1,\ldots,h_s)$ be the $h$-vector of the broken circuit complex of a matroid $M$. Set $h'_i=h_i/\binom{h_1+i-1}{i}$ for $i=0,1,\ldots, s$. Then
\begin{enumerate}
 \item $h(M)$ is unimodal.

 \item $h(M)$ is log-concave.

 \item $h(M)$ is strongly log-concave, i.e., the sequence $(h'_0,h'_1,\ldots,h'_s)$ is log-concave.
\end{enumerate}
 \end{conjecture}

This still wide open conjecture was proposed by Brylawski \cite[p. 232]{Bry82}. Therein, he also showed that \Cref{con2}(ii) is stronger than Rota--Heron's conjecture \cite{Ro71, He72} and Welsh's conjecture \cite{We76}. As we mentioned before, significant progress towards proving \Cref{con2}(ii) was made by Huh \cite{Hu15}, who verified it for matroids representable over a field of characteristic zero. 

Concerning \Cref{con} it is also worth noting the following question:

\begin{question}
Let $M$ be a matroid and let $h(M)=(h_0,h_1,\ldots,h_s)$ be the $h$-vector of $BC(M)$, where $h_s\neq 0$. Define $g(M)=(1,h_1-h_0,\ldots,h_{\lfloor s/2\rfloor}-h_{\lfloor s/2\rfloor -1})$ to be the $g$-{\rm vector} of $BC(M)$. Is it always true that $g(M)$ is an $O$-sequence?
\end{question}

This question together with \Cref{con} was posed by Swartz in \cite{Sw03}, where he gave an affirmative answer to the question in the case of independence complexes. We believe that this question should also have an affirmative answer for broken circuit complexes in general. However, we would like to remark that it is not clear whether the question can be reduced to the case of parallel irreducible matroids. For this, one would, in analogy with \Cref{red}, need that the property of the $g$-vector being an $O$-sequence is preserved under taking products. Currently, in joint work with Uwe Nagel, the first author is investigating this problem. 

\begin{remark}(This remark is not contained in the published version of the paper.)
At a recent workshop in Oberwolfach (from 11 to 17 December 2016), June Huh informed the second author that he and his coauthors had resolved \Cref{con2}(ii) by mimicking their method applied to $f$-vectors in \cite{AHK}. Thus \Cref{con} now holds true in its full generality.
\end{remark}

\section*{Acknowledgments} 

We wish to thank Ed Swartz for pointing out that \Cref{con2} is due to Brylawski. We are also grateful to the referees, whose thoughtful comments and useful suggestions led to several improvements in the paper. The first author was supported by the German Research Council DFG-GRK~1916.


\begin{thebibliography}{99}

\bibitem{AHK}
K. Adiprasito, J. Huh and E. Katz, \emph{Hodge theory for combinatorial geometries}, \href{http://arxiv.org/abs/1511.02888}{arXiv:1511.02888}.


\bibitem{Bj92}
A. Bj\"{o}rner, \emph{The homology and shellability of matroids and geometric lattices}, in \emph{Matroid Applications}, pp. 226--283, Encyclopedia Math. Appl. {\bf 40}, Cambridge University Press, Cambridge, 1992.

\bibitem{BMM+} M. Boij, J. Migliore, R. Mir\'{o}-Roig, U. Nagel and F. Zanello, \emph{On the shape of a pure $O$-sequence}, Mem. Amer. Math. Soc. {\bf 218} (2012), no. 1024.

\bibitem{BP09}
T. Braden and N. Proudfoot, \emph{The hypertoric intersection cohomology ring}, Invent. Math. {\bf 177} (2009), no. 2, 337--379.


\bibitem{Bry71}
T. Brylawski, \emph{A combinatorial model for series--parallel networks}, Trans. Amer. Math. Soc. {\bf 154} (1971), 1--22.

\bibitem{Bry77}
T. Brylawski, \emph{The broken-circuit complex}, Trans. Amer. Math. Soc. {\bf 234} (1977), 417--433.


\bibitem{Bry82}
T. Brylawski, \emph{The Tutte polynomial. I. General theory}, in \emph{Matroid theory and its applications}, 125--275, Liguori, Naples, 1982.


\bibitem{BO92}
T. Brylawski and J. Oxley,
\emph{The Tutte polynomial and its applications}, in \emph{Matroid Applications}, 123--225, Encyclopedia Math. Appl. {\bf 40}, Cambridge University Press, Cambridge, 1992.

\bibitem{Ch97}
M. K. Chari, \emph{Two decompositions in topological combinatorics with applications to matroid complexes}, Trans. Am. Math. Soc. {\bf349} (1997), 3925--3943.


\bibitem{Cr67}
H. Crapo, \emph{A higher invariant for matroids}, J. Combinatorial Theory {\bf 2} (1967), 406--417.


\bibitem{DGT14}
G. Denham, M. Garrousian and  S. Toh\v{a}neanu, \emph{Modular decomposition of the Orlik--Terao algebra}, Ann. Comb.{\bf 18} (2014), no. 2, 289--312.


\bibitem{EPY03}
D. Eisenbud, S. Popescu and S. Yuzvinsky, \emph{Hyperplane arrangement cohomology and monomials in the exterior algebra}, Trans. Amer. Math. Soc. {\bf 355} (2003), 4365--4383.

%\bibitem{Ha05}
%T. Hausel, \emph{Quaternionic geometry of matroids}, Cent. Eur. J. Math. {\bf 3} (2005), no. 1, 26--38.

\bibitem{He72}
A. P. Heron, \emph{Matroid polynomials}, in \emph{Combinatorics} (Proc. Conf. Combinatorial Math., Math. Inst., Oxford, 1972), pp. 164--202, Inst. Math. Appl., Southend-on-Sea, 1972. 

\bibitem{Hi89}
T. Hibi, \emph{What can be said about pure $O$-sequences?}, J. Combin. Theory Ser. A {\bf 50} (1989), 319--322.

\bibitem{Hi92}
T. Hibi, \emph{Face number inequalities for matroid complexes and Cohen--Macaulay types of Stanley--Reisner rings of distributive lattices}, Pacific J. Math. {\bf 154} (1992), 253--264.


\bibitem{Hu15}
J. Huh, \emph{$h$-vectors of matroids and logarithmic concavity}, Adv. Math. {\bf 270} (2015), 49--59.


\bibitem{KR09}
G. K\"{a}mpf and  T. R\"{o}mer, \emph{Homological properties of Orlik--Solomon algebras}, Manuscripta Math. {\bf 129} (2009), 181--210.

\bibitem{Le14}
D. V. Le, \emph{On the Gorensteinness of broken circuit complexes and Orlik--Terao ideals},  J. Combin. Theory Ser. A {\bf 123} (2014), no. 1, 169--185.

\bibitem{Le16}
D. V. Le, \emph{Broken circuit complexes of series--parallel networks}, European J. Combin. {\bf 51} (2016), 12--36.

\bibitem{LM15}
D. V. Le and F. Mohammadi, \emph{On the Orlik--Terao ideal and the relation space of a hyperplane arrangement}, Adv. in Appl. Math. {\bf 71} (2015), 34--51.

\bibitem{LR13}
D. V. Le and T. R\"{o}mer, \emph{Broken circuit complexes and hyperplane arrangements}, J. Algebraic Combin. {\bf 38} (2013), no. 4, 989--1016.

\bibitem{MP15}
M. McBreen and N. Proudfoot, \emph{Intersection cohomology and quantum cohomology of conical symplectic resolutions}, Algebr. Geom. {\bf 2} (2015), no. 5, 623--641.


\bibitem{OT94}
P. Orlik and  H. Terao, \emph{Commutative algebras for arrangements}. Nagoya Math. J. {\bf 134} (1994), 65--73.

\bibitem{Ox11}
J. Oxley, \emph{Matroid theory}, 2nd ed., Oxford Graduate Texts in Mathematics {\bf 21}, Oxford
University Press, Oxford, 2011.

\bibitem{PS06}
N. Proudfoot and D. Speyer, \emph{A broken circuit ring}, Beitr\"{a}ge Algebra Geom. {\bf 47} (2006), no. 1, 161--166.

\bibitem{Pr77}
J. S. Provan, \emph{Decompositions, shellings, and diameters of simplicial complexes and convex polyhedra}, Thesis, Cornell University, Ithaca, NY, 1977.


\bibitem{Ro64}
G.-C. Rota, \emph{On the foundations of combinatorial theory. I. Theory of M\"{o}bius functions}, Z. Wahrscheinlichkeitstheorie and Verw. Gebiete {\bf 2} (1964), 340--368.

\bibitem{Ro71}
G.-C. Rota, \emph{Combinatorial theory, old and new}, in: Proc. Internat. Cong. Math. (Nice, 1970), Gauthier-Villars, Paris, 1971, pp. 229--233.

\bibitem{SSV13}
R. Sanyal,  B. Sturmfels and C. Vinzant, \emph{The entropic discriminant}, Adv. Math. {\bf 244} (2013), 678--707.

\bibitem{Sc11}
H. Schenck, \emph{Resonance varieties via blowups of $\mathbb{P}^2$ and scrolls}, Int. Math. Res. Not. {\bf 20} (2011), 4756--4778.

\bibitem{ST09}
H. Schenck and \c{S}. Toh\v{a}neanu, \emph{The Orlik--Terao algebra and $2$-formality}, Math. Res. Lett. {\bf 16} (2009), 171--182.



\bibitem{Sta77}
R. P. Stanley, \emph{Cohen--Macaulay complexes}, in \emph{Higher combinatorics}, pp. 51--62, Reidel, Dordrecht, 1977.

\bibitem{Sta89}
R. P. Stanley, \emph{Log-concave and unimodal sequences in algebra, combinatorics, and geometry}, in \emph{Graph theory and its applications} {\bf 576}, pp. 500--535, 1989.
\bibitem{Sw03}
E. Swartz, \emph{$g$-elements of matroid complexes}, J. Combin. Theory Ser. B {\bf 88} (2003), no. 2, 369--375.


\bibitem{Sw04} 
E. Swartz, \emph{Lower bounds for $h$-vectors of $k$-CM, independence, and broken circuit complexes},
SIAM J. Discrete Math. {\bf 18} (2004/05), no. 3, 647--661. 

 \bibitem{Sw06} 
E. Swartz, \emph{$g$-Elements, finite buildings and higher Cohen--Macaulay connectivity},  J. Combin. Theory Ser. A {\bf 113} (2006), no. 7, 1305--1320. 

\bibitem{Te02}
H. Terao, \emph{Algebras generated by reciprocals of linear forms}, J. Algebra {\bf 250} (2002), 549--558.

\bibitem{We76}
D. Welsh, \emph{Matroid theory}, Academic Press, London, 1976.

\bibitem{Wh32}
H. Whitney, \emph{A logical expansion in mathematics}, Bull. Amer. Math. Soc. {\bf 38} (1932), 572--579.

\bibitem{Wh35}
H. Whitney, \emph{On the abstract properties of linear dependence}, Amer. J. Math. {\bf 57} (1935), no. 3, 509--533.


\bibitem{Za87}
T. Zaslavsky, \emph{The M\"{o}bius function and the characteristic polynomial}, in \emph{Combinatorial geometries}, pp. 114--138, Encyclopedia Math. Appl. {\bf 29}, Cambridge University Press, Cambridge, 1987.

\end{thebibliography}
\end{document}